\documentclass[a4paper, 12pt, reqno]{amsart}
\usepackage{amsfonts, amsmath, amssymb, amsthm}
\usepackage[english]{babel}
\usepackage[colorlinks=true, linkcolor=blue, citecolor=blue]{hyperref}
\usepackage[utf8]{inputenc}
\usepackage{tikz-cd}
\usepackage[outline]{contour}
\contourlength{1.2pt}
\advance\hoffset-20mm \advance\textwidth40mm
\renewcommand{\AA}{\mathbb{A}}

\newcommand{\GG}{\mathbb{G}}
\newcommand{\KK}{\mathbb{K}}
\newcommand{\CC}{\mathbb{C}}

\newcommand{\PP}{\mathbb{P}}

\newcommand{\SAut}{\mathrm{SAut}}
\newcommand{\Aut}{\mathrm{Aut}}

\newcommand{\GL}{\mathrm{GL}}
\newcommand{\SL}{\mathrm{SL}}

\newcommand{\Gr}{\mathrm{Gr}}
\theoremstyle{definition}
\newtheorem{example}{Example}
\newtheorem{remark}{Remark}
\newtheorem{definition}{Definition}
\newtheorem{problem}{Problem}
\theoremstyle{plain}
\newtheorem{corollary}{Corollary}
\newtheorem{lemma}{Lemma}
\newtheorem{proposition}{Proposition}
\newtheorem{theorem}{Theorem}
\newtheorem*{TheoremA}{Theorem A}
\newtheorem*{TheoremB}{Theorem B}
\newtheorem*{TheoremC}{Theorem C}
\begin{document}
\sloppy 
\title[On images of affine spaces]{On images of affine spaces}
\author{Ivan Arzhantsev}
\address{Faculty of Computer Science, HSE University, Pokrovsky Boulevard 11, Moscow, 109028 Russia}
\email{arjantsev@hse.ru}
\thanks{Supported by the grant RSF-DST 22-41-02019}
\begin{abstract}
We prove that every non-degenerate toric variety, every homogeneous space of a connected linear algebraic group without non-constant invertible regular functions, and every variety covered by affine spaces admits a surjective morphism from an affine space. 
\end{abstract} 
\maketitle

\section{$A$-images and very flexible varieties}

There is no doubt that the affine spaces $\AA^m$ play the key role in mathematics and other fields of science. It is all the more surprising that despite the centuries-old history of study, to this day a number of natural and even naive questions about affine spaces remain open. Here we can mention the Jacobian conjecture, the cancellation problem, the linearization problem, the problem of recognition of tame and wild automorphisms, and a number of other unsolved problems. The question of which varieties can be realized as images of morphisms from an affine space to other varieties also seems nontrivial. The purpose of this note is to show that the class of such images is surprisingly wide and entirely includes several important classes of algebraic varieties. Our results are based on the theory of flexible varieties developed in~\cite{AFKKZ} and subsequent works, and on a canonical quotient realization of a variety coming from the theory of Cox rings, see~\cite{Cox,ADHL}. 

Below we consider algebraic varieties over an algebraically closed field $\KK$ of characteristic zero. 

\begin{definition}
An algebraic variety $X$ is called an \emph{$A$-image} if for some positive integer $m$ there is a surjective morphism $\varphi\colon \AA^m\to X$. 
\end{definition} 

Clearly, every $A$-image is irreducible. Moreover, if $f$ is an invertible regular function on $X$, the pull-back function $\varphi^*(f)$ is an invertible regular function on $\AA^m$. This proves that every invertible regular function on an $A$-image is constant. Finally, let us recall that an irreducible algebraic variety $X$ is \emph{unirational} if the field of rational functions $\KK(X)$
can be embedded in a purely transcendental extension $\KK(x_1,\ldots,x_m)$ of the ground field $\KK$, or, equivalently, there exists a rational dominant morphism from some affine space $\AA^m$ to $X$. It follows that every $A$-image is a unirational variety. Honestly speaking, we have no example of a unirational variety without non-constant invertible regular functions which is not an $A$-image. Moreover, explicit computations show that many varieties from this class can be realized as images of affine spaces. 

\begin{example} \label{ex1}
There is a surjective morphism $\varphi\colon\AA^1\to\PP^1$. It can be given, for example, by 
$$
x\to [1+x^2 : x].
$$
\end{example}

\begin{example} \label{ex2}
It is easy to check that the image of the morphism $\AA^2\to\AA^2$ given by
$$
(x,y)\to (1+xy, x+y^2+xy^3) 
$$
is $\AA^2\setminus\{0\}$. So we obtain a surjective morphism $\varphi\colon\AA^2\to\AA^2\setminus\{0\}$. Let us note that in~\cite[Section~1]{LT} the authors were interested in the existence of such a morphism. 
\end{example} 

\begin{remark} 
As we learned from~\cite{Ba}, Zbigniew Jelonek has constructed such and more general examples much earlier, see~\cite{Je}. In~\cite[Proposition~1.5]{Ba}, one may find an explicit formulas for a morphism $\CC^n\to \CC^n$, $n\ge 2$, such that the complement of the image is a given finite set of points. 
\end{remark} 

Since there is a quotient morphism $\AA^2\setminus\{0\}\to\PP^1$, Example~\ref{ex2} provides one more way to obtain the projective line $\PP^1$ as an image of affine space. In the present note we generalize this approach in order to obtain many varieties as $A$-images. 

\smallskip

Let us proceed with the following obvious observations. 

\begin{lemma}
\begin{enumerate}
\item[a)]
If $\psi\colon X_1\to X_2$ is a surjective morphism of algebraic varieties and $X_1$ is an $A$-image, then $X_2$ is an $A$-image. 
\item[b)] 
If $X_1$ and $X_2$ are $A$-images, then the direct product $X_1\times X_2$ is an $A$-image. 
\item[c)]
An irreducible curve $C$ is an $A$-image if and only if the normalization of $C$ is $\PP^1$ or~$\AA^1$. In these cases $C$ is an image of the line $\AA^1$. 
\end{enumerate}
\end{lemma}

It is useful to observe that an affine variety $X$ is an $A$-image if and only if the algebra of regular functions $\KK[X]$ can be embedded into some polynomial algebra $\KK[x_1,\ldots,x_m]$ is such a way that any proper ideal in $\KK[X]$ generates a proper ideal in $\KK[x_1,\ldots,x_m]$. 

It was observed by the referee that the concept of an $A$-image is natural in the context of the famous Zariski Cancellation problem, which asks whether an affine variety $X$ such that the direct product $X\times\AA^k$ is isomorphic to the affine space $\AA^m$ is necessarily isomorphic to the affine space $\AA^{m-k}$. Any potentional counterexample to this problem is an affine variety, which is non-isomorphic to an affine space and is an image of $\AA^m$ under the projection to the first component in the direct product decomposition. 

Let us mention one geometric property that all $A$-images have. It is well known that for any $m\ge 2$, any positive integer $k$ and any two $k$-tuples of pairwise distinct points on $\AA^m$ there is an automorphism of $\AA^m$ which sends the first tuple to the second one. This implies that for any finite subset $F$ in $\AA^m$ there is a curve $C$ in $\AA^m$ isomorphic to $\AA^1$ that contains $F$; cf.~\cite[Corollary~4.18]{AFKKZ}. We conclude that every $A$-image $X$ is a \emph{strongly $\AA^1$-connected} variety, i.e., for every finite subset $F$ in $X$ there is a morphism $\AA^1\to X$ whose image contains $F$. At the same time it is known that for any finine subset $F$ in a complete unirational variety $X$ there exists a moprhism from $\PP^1$ (and thus from $\AA^1$) to $X$, whose image contains $F$; see~\cite[IV.3.9]{Ko}. 

\smallskip 

Now we come to a central concept of this note. Denote by $\GG_a$ the additive group $(\KK,+)$ of the ground field. Any nontrivial regular action $\GG_a\times X\to X$ on an algebraic variety $X$ gives rise to a subgroup $H$ in the automorphism group $\Aut(X)$, which we call a \emph{$\GG_a$-subgroup}. Let $\SAut(X)$ be the subgroup in $\Aut(X)$ generated by all $\GG_a$-subgroups. 

\begin{definition}
An irreducible variety $X$ is called \emph{very flexible} if the group $\SAut(X)$ acts on $X$ transitively.
\end{definition}

Any very flexible variety is smooth. By \cite[Theorem~0.1]{AFKKZ}, \cite[Theorem~2]{APS}, and \cite[Theorem~1.11]{FKZ}, for a smooth quasi-affine variety the condition to be very flexible is equivalent to flexibility in the sence of definitions given in these papers. 

\begin{proposition} \cite[Proposition~1.4]{AFKKZ} \label{Pr1} 
Let $X$ be a very flexible variety. Then there are (not necessarily distinct) $\GG_a$-subgroups $H_1,\ldots,H_m$ in $\Aut(X)$ such that
$$
X=(H_1\cdot\ldots\cdot H_m). x
$$
for any point $x\in X$. 
\end{proposition} 

The next proposition is borrowed from~\cite[Corollary~1.11]{AFKKZ}. 

\begin{proposition} \label{Pr2}
Every very flexible variety is an $A$-image.
\end{proposition}

\begin{proof}
Fix a point $x\in X$ and take a sequence of $\GG_a$-subgroups $H_1,\ldots,H_m$ as in Proposition~\ref{Pr1}. Then the map 
$$
H_1\times\ldots\times H_m \to (H_1\cdot \ldots \cdot H_m). x, \quad (h_1,\ldots,h_m)\mapsto h_1\ldots h_m.x 
$$
is a surjective morphism from $\AA^m$ to $X$. 
\end{proof} 

Below we show that Proposition~\ref{Pr2} allows to realize many varieties as $A$-images. The next result follows directly from~\cite[Theorem~0.1]{FKZ}. 

\begin{theorem} \label{thvf}
Let $X$ be a very flexible quasi-affine variety and $Z\subseteq X$ be a closed subvariety of codimension at least $2$. Then the variety $X\setminus Z$ is very flexible. 
\end{theorem} 

\begin{corollary} \label{cor}
Let $Z$ be a closed subset in $\AA^n$ of codimension at least $2$. Then the variety $\AA^n\setminus Z$ is very flexible.  
\end{corollary} 

\section{Three classes of $A$-images}

\subsection{Toric varieties} Let $T$ be an algebraic torus. A normal irreducible variety X is called \emph{toric}, if there is a faithful action of $T$ on $X$ with an open orbit.
We refer to~\cite{CLS,Fu} for a systematic theory of toric variety. 

A toric variety $X$ is called \emph{degenerate} if it is isomorphic to a direct product $Y\times\KK^{\times}$ for some toric variety $Y$. By~\cite[Proposition~3.3.9]{CLS}, a toric variety $X$ is degenerate if and only if there exists a non-constant invertible regular function on $X$. It is easy to see that any toric variety $X$ is isomorphic to a direct product $X'\times (\KK^{\times})^k$ for some non-negative integer~$k$, where $X'$ is a non-degenerate toric variety. Clearly, a toric variety can be an $A$-image only if $X$ is non-degenerate.  

\begin{TheoremA} \label{ThA}
A toric variety $X$ is an $A$-image if and only if $X$ is non-degenerate. 
\end{TheoremA}

\begin{proof}
It remains to prove that a non-degenerate toric variety $X$ is an $A$-image. By~\cite[Theorem~2.1]{Cox}, the variety $X$ can be realized as a good quotient by an action of a diagonalizable group on the open subset $\AA^l\setminus Z$, where $Z$ is (maybe empty) union of some coordinate planes of codimension at least $2$ in $\AA^l$. By Corollary~\ref{cor}, the variety 
$\AA^l\setminus Z$ is an $A$-image. So $X$ is an $A$-image as well. 
\end{proof} 

\begin{remark}
It follows from~\cite[Theorem~1.6]{Fo} that for every compact smooth complex toric variety $X$ of dimension $n$ there exists a surjective morpshim $\varphi\colon \CC^n\to X$. The proof is based on an algebraic version of a theory of subelliptic manifolds. The technique used in this work is new to us, and we failed to construct explicitly the corresponding surjective morphisms. But this result motivated us to write the present note.
\end{remark}

\subsection{$A$-covered varieties} Let us recall from \cite[Definition~4]{APS} that an irreducible algebraic variety $X$ is said to be \emph{$A$-covered} if there is an open covering 
$X=U_1\cup \ldots\cup U_r$, where every chart $U_i$ is isomorphic to the affine space $\AA^n$. The class of $A$-covered varieties includes smooth complete spherical varieties, smooth projective rational surfaces, smooth complete rational varieties with a torus action of complexity~1, and many other varieties; see~\cite[Section~4]{APS}. 

\smallskip

Let us give non-tivial examples of affine $A$-covered varieties.

\begin{example}
Fix an integer $k\ge 1$ and consider the surface $X=V(x_1^2-x_2^kx_3-1)$ in~$\AA^3$. Take the divisor $D^+$ on $X$ given by $x_1=1,x_2=0$ and the divisor $D^-$ on $X$ given by $x_1=-1,x_2=0$. Clearly, $D^+$ and $D^-$ have empty intersection. The subset $U^+:=X\setminus D^+$ is the complement to a divisor on a smooth affine variety, so $U^+$ is affine; see~\cite[Lemma~3.3]{Na-2}. Similarly, the open subset $U^-:=X\setminus D^-$ is affine as well. 

Consider the rational function 
$$
f=\frac{x_3}{x_1-1}=\frac{x_1+1}{x_2^k}. 
$$
It is clear that $f$ is regular on $U^+$ and $x_1=x_2^kf-1$, $x_3=(x_1-1)f$. It follows that the algebra of regular functions $\KK[U^+]$ is the polynomial algebra with generators $x_2$ and $f$. Similarly, the algebra $\KK[U^-]$ is the polynomial algebra with generators $x_2$ and $g$, where
$$
g=\frac{x_3}{x_1+1}=\frac{x_1-1}{x_2^k}. 
$$
We conclude that $X$ is covered by two open charts $U^+$ and $U^-$ each isomorphic to $\AA^2$. So the surface $X$ is $A$-covered and is not isomorphic to $\AA^2$.
\end{example} 

\begin{remark}
The referee drew our attention to the fact that the idea of the computation performed in the above example is contained in the unpublished work~\cite{Da}. Moreover, the referee proposed the following generalization of this construction. Let $X$ be a smooth affine surface endowed with a smooth surjective $\AA^1$-fibration $\varphi\colon X\to\AA^1$. Miyanishi proved that every scheme-theoretic fiber of $\varphi$, which is not isomorphic to $\AA^1$, decomposes as a disjoint union of curves all isomorphic to $\AA^1$. Removing from each such fiber all such curves but one, we obtain an open affine subset in $X$ with a smooth surjective $\AA^1$-fibration over $\AA^1$ with irreducible fibers. By~\cite[Theorem~1]{KM}, such an affine open subset is isomorphic to $\AA^2$. We conclude that the surface $X$ is $A$-covered.  
\end{remark} 

\begin{TheoremB} \label{ThB}
Every $A$-covered variety is an $A$-image. 
\end{TheoremB}

\begin{proof}
By~\cite[Theorem~3]{APS}, an $A$-covered variety $X$ is a geometric quotient of a very fllexible quasi-affine variety by an action of a torus. So the claim follows from Proposition~\Ref{Pr2}. 
\end{proof} 

\begin{corollary}
Every complete rational surface is an $A$-image.
\end{corollary}

\begin{proof}
Passing to normalization and desingularization, we may assume that we deal with a smooth projective rational surface. The latter is $A$-covered; see \cite[Section~4]{APS}. 
\end{proof} 

\subsection{Homogeneous spaces} Let $X$ be a homogeneous space $G/H$, where $G$ is a connected linear algebraic group and $H$ is a closed subgroup in $G$. 

\begin{TheoremC} \label{ThC} 
A homogeneous space $X$ is an $A$-image if and only if $\KK[X]^{\times}=\KK^{\times}$. 
\end{TheoremC}

\begin{proof}
We already know that an $A$-image has no non-constant invertible regular function. Conversely, assume that $\KK[X]^{\times}=\KK^{\times}$. By~\cite[Lemma~5.1]{ASZ}, in this case
the variety $X$ is homogeneous with respect to the subgroup $G^{\text s}$ in $G$, which is a semidirect product of a maximal semisimple subgroup and the unipotent radical of $G$. Since the group $G^{\text s}$ is generated by $\GG_a$-subgroups (see, e.g., \cite[Lemma~1.1]{Po}), we conclude that $X$ is very flexible and hence is an $A$-image. 
\end{proof} 

\begin{remark}
Since a homogeneous space of the group $\SL(V)$ need not be rational (see, e.g., \cite[Example~1.22]{Po}), we conclude that an $A$-image is not always rational. 
\end{remark}

\begin{corollary} \label{cnfg}
There is a quasi-affine $A$-image $X$ with a non-finitely generated algebra of regular functions $\KK[X]$.  
\end{corollary}

\begin{proof}
Let $H\subseteq\GL(V)$ be a unipotent subgroup such that the algebra of invariants $\KK[V]^H$ is not finitely generated. Examples of such subgroups are known since the work of  Nagata~\cite{Na}; they provide counterexamples to Hilbert's Fourteenth Problem. It is proved in~\cite{Gr} that in this case the algebra of regular functions on the quasi-affine homogeneous space $X:=\SL(V)/H$ is not finitely generated, while Theorem~C implies that $X$ is an $A$-image. 
\end{proof}

\section{Concluding remarks and problems}

We say that a morphism $\pi\colon X'\to X$ of irreducible algebraic varieties \textit{does not contract divisors} if the image in $X$ of any prime divisor $D$ on $X'$ is not contained in a closed subvariety of codimension $2$ in $X$. 

\begin{proposition} \label{prcod2}
Let $\pi\colon X'\to X$ be a surjective morphism of irreducible varieties which does not contract divisors and the variety $X'$ be quasi-affine and very flexible. If $Y\subseteq X$ is a closed subvariety of codimension at least $2$, then the complement $X\setminus Y$ is an $A$-image.
\end{proposition} 

\begin{proof}
By assumptions, the preimage $\pi^{-1}(Y)$ does not contain divisors on $X'$. By Theorem~\ref{thvf}, the variety $X'\setminus\pi^{-1}(Y)$ is very flexible and hence is an $A$-image. So 
$X\setminus Y$ is an $A$-image as the image of $X'\setminus\pi^{-1}(Y)$ under the restriction of the morphism $\pi$. 
\end{proof}

\begin{remark}
Note that Proposition~\ref{prcod2} is applicable to varieties from Theorems~A-C. More precisely, for a non-degenerate toric variety $X$ the good quotient $\pi\colon\AA^l\setminus Z\to X$ does not contract divisors by \cite[Proposition~1.6.1.6.(ii)]{ADHL} and the variety $\AA^l\setminus Z$ is very flexible by Corollary~\ref{cor}. For $A$-covered varieties, the desired morphism $\pi\colon X'\to X$ is provided by~\cite[Theorem~3]{APS}. Finally, for a homogeneous space $G/H$ such that the subgroup $G^{\text s}$ in $G$ acts on $G/H$ transitively, the orbit map
$\pi\colon G^{\text s}\to G/H$ is surjective, fibers of this morphism are pairwise isomorphic, so it does not contract divisors, and the variety $G^{\text s}$ is affine and very flexible. 
\end{remark}

Continuing the line of Corollary~\ref{cnfg} with exotic examples of $A$-images, we prove the following result. 

\begin{proposition}
There exists a three-dimensional complete $A$-image $X$ that is not embeddable into any toric variety. 
\end{proposition}

\begin{proof}
It is proved in \cite[Theorem~A]{Wl} that an irreducible normal variety $X$ admits a closed embedding into a toric variety if and only if every pair of points on $X$ is contained in a common affine neighborhood. In~\cite[Example~6.4]{Sw} the author constructs a good quotient $U\to X$ by an action of a one-dimensional torus $T$, where $U$ is an open subset in the Grassmannian $\Gr(4,2)$, the quotient space $X$ is complete, and there are two points on $X$ which are not contained in a common affine neighborhood. It follows that $X$ is three-dimensinal and it is not embeddable into any toric variety.

It is clear from the construction in~\cite[Example~6.4]{Sw} that the complement $\Gr(4,2)\setminus U$ has codimension at least $2$ in $\Gr(4,2)$. Since the Grassmannian $\Gr(4,2)$ is a homogeneous space of the group $\SL_4$, Theorem~C and Proposition~\ref{prcod2} imply that $X$ is an $A$-image.
\end{proof} 

\begin{corollary}
An $A$-image need not be quasiprojective. 
\end{corollary}

Now we would like to formulate some problems related to the subject of this note. 

\begin{problem}
Find necessary and sufficient conditions for an algebraic variety $X$ to be an $A$-image. 
\end{problem}

One can also put a more special problem. 

\begin{problem} \label{prob}
Characterize $n$-dimensional algebraic varieties $X$ such that there exists a surjective morphism $\varphi\colon\AA^n\to X$. 
\end{problem}

In fact, we have no example of an $n$-dimensional $A$-image $X$ such that there is no surjective morphism $\varphi\colon\AA^n\to X$. In particular, a variety $X$ in Problem~\ref{prob} need not be rational: there are examples of quotient spaces $X=\AA^n/\!/H$, where $H$ is a finite group and $X$ is not rational; see \cite[Theorem~3.6]{Sa}. 

\begin{remark}
In the preliminary version of this note we also asked whether for any irreducible algebraic variety $X$ there exists a surjective morphism $\phi\colon V\to X$ from an irreducible affine variety $V$. The referee informed us that the answer to this question is positive. Namely, Jouanolou~\cite{Jo} proved that every quasi-projective variety $X$ admits a vector bundle torsor $V\to X$ with affine total space $V$. This result was extended by Thomason (see~\cite{We}) to a wider class of algebraic varieties including all smooth algebraic varieties. Applying Hironaka's desingularization theorem, we obtain a desingularization $\widetilde{X}\to X$ and then use Jouanolou-Thomason's construction $\widetilde{V}\to \widetilde{X}$ to get a surjection $\widetilde{V}\to X$.  Moreover, the referee observed that $X$ is an $A$-image if and only if any affine variety $\widetilde{V}$ obtained this way is an $A$-image. 
\end{remark}

The next question asks whether we may use $A$-images as ambient spaces for arbitrary algebraic varieties, just as affine (projective) spaces serve for affine (projective) varieties.  

\begin{problem}
Let $X$ be an algebraic variety. Is it possible to realize $X$ as a closed subvariety in some $A$-image $\widetilde{X}$?
\end{problem}

Finally, let us consider an even more general situation. It is well known that the image of a morphism $\varphi\colon X_1\to X_2$ of algebraic varieties need not be a subvariety in $X_2$. By Chevalley's theorem, the image of a morphism is a constructible subset of $X_2$, i.e., a finite disjoint union of locally closed subvarieties in $X_2$; see~\cite[Theorem~I.4.4]{Hu}. 

\begin{problem}
Let $W$ be an irreducible constructible subset of an algebraic variety $X_2$. Do there exist an irreducible algebraic variety $X_1$ and a morphism $\varphi\colon X_1\to X_2$ such that the image of $\varphi$ coincides with~$W$? 
\end{problem}

\begin{remark}
After a preprint version of this note appeared in September, 2022, three more works~\cite{Ba,KZ,Ku} with new results on $A$-images became available. 
\end{remark}

\smallskip

\textit{Acknowledgements}. The author is grateful to Yuri Prokhorov and Constantin Shramov for helpful consultations and references, and to Viktor Balch Barth for useful e-mail correspondence. Special thanks are due to the referee for deep observations and comments that clarify many points related to this research. 

{}

\begin{thebibliography}{}
%
\bibitem{ADHL}
Ivan Arzhantsev, Ulrich Derenthal, J\"urgen Hausen, and Antonio Laface. \textit{Cox rings}. Cambridge Studies in Advanced Mathematics~144, Cambridge University Press, Cambridge, 2015
%
\bibitem{AFKKZ} 
Ivan Arzhantsev, Hubert Flenner, Shulim Kaliman, Frank Kutzschebauch, and Mikhail Zaidenberg. Flexible varieties and automorphism groups. Duke Math. J. 162 (2013), no.~4, 767-823
%
\bibitem{APS}
Ivan Arzhantsev, Alexander Perepechko, and Hendrik S\"uss. Infinite transitivity on universal torsors. J.~London Math. Soc. 89 (2014), no.~3, 762-778
%
\bibitem{ASZ}
Ivan Arzhantsev, Kirill Shakhmatov, and Yulia Zaitseva. Homogeneous algebraic varieties and transitivity degree. Proc. Steklov Inst. Math. 318 (2022), 13-25 
%
\bibitem{Ba}
Viktor Balch Barth. Surjective morphisms from affine space to its Zariski open subsets. arXiv:2302.11470, 7 pages 
%
\bibitem{Cox}
David Cox. The homogeneous coordinate ring of a toric variety. J. Algebraic Geom. 4 (1995), no.~1, 17-50
%
\bibitem{CLS}
David Cox, John Little, and Henry Schenck. \textit{Toric Varieties}. Graduate Studies in Math.~124, Amer. Math. Soc., Providence, RI, 2011
%
\bibitem{Da}
Wlodzimierz Danielewski. On a cancellation problem and automorphism groups of algebraic varieties. Preprint, Warsaw, 1989
%
\bibitem{FKZ} 
Hubert Flenner, Shulim Kaliman, and Mikhail Zaidenberg. The Gromov--Winkelmann theorem for flexible varieties. J.~Eur. Math. Soc. 18 (2016), no.~11, 2483-2510
%
\bibitem{Fo}
Frank Forstneri\v{c}. Surjective holomorphic maps onto Oka manifolds. In: Complex and Symplectic Geometry, Springer INdAM Ser., vol.~21, Springer, Cham, 2017, pp. 73-84
%
\bibitem{Fu}
William Fulton. \textit{Introduction to Toric Varieties}. Annales of Math. Studies~131, Princeton University Press, Princeton, NJ, 1993
%
\bibitem{Gr}
Frank Grosshans. Observable groups and Hilbert's fourteenth problem. Amer. J. Math. 95 (1973), no.~1, 229-253 
%
\bibitem{Hu}
James Humphreys. \textit{Linear Algebraic Groups}. Graduate Texts Math. 21, Springer Verlag, New York, 1975
%
\bibitem{Je}
Zbigniew Jelonek. A number of points in the set $\CC^2\setminus F(\CC^2)$. Bull. Polish Acad. Sci. Math. 47 (1999), no.~3, 257-261
%
\bibitem{Jo}
Jean-Pierre Jouanolou. Une Suite exact de Mayer-Vietoris en K-Theorie Algebrique. In Algebraic K-theory, I: Higher K-theories. Lecture Notes in Math., vol.~341, Springer, Berlin, 1973, pp.~293-316 
%
\bibitem{KZ}
Shulim Kaliman and Mikhail Zaidenberg. Gromov ellipticity of cones over projective manifolds. arXiv:2303.02036, 21 pages 
%
\bibitem{KM}
Tatsuji Kambayashi and Masayoshi Miyanishi. On flat fibrations by the affine line. Illinois J. Math. 22 (1978), no.~4, 662-671 
%
\bibitem{Ko}
J\'anos Koll\'ar. \textit{Rational Curves on Algebraic Varieties}. Ergeb. Math. Grenzgeb. (3)~32, Springer-Verlag, Berlin, 1996
%
\bibitem{Ku}
Yuta Kusakabe. Surjective morphisms onto subelliptic varieties. arXiv:2212.06412, 7 pages 
%
\bibitem{LT}
Finnur L\'arusson and Tuyen Trung Truong. Approximation and interpolation of regular maps from affine varieties to algebraic manifolds. Math. Scand. 125 (2019), no.~2, 199-209 
%
\bibitem{Na}
Masayoshi Nagata. On the 14th problem of Hilbert. Amer. J. Math. 81 (1959), 766-772
%
\bibitem{Na-2}
Masayoshi Nagata. Note on orbit spaces. Osaka Math. J. 14 (1962), 21-31
%
\bibitem{Po}
Vladimir Popov. On the Makar-Limanov, Derksen invariants, and finite automorphism groups of algebraic varieties. CRM Proc. Lecture Notes, vol.~54, Amer. Math. Soc., Providence, RI, 2011, pp. 289-311
%
\bibitem{Sa}
David Saltman. Noether's problem over an algebraically closed field. Invent. Math. 77 (1984), 71-84
%
\bibitem{Sw}
Joanna \'Swi\c{e}cicka. A combinatorial construction of sets with good quotients by an action of a reductive group. Colloq. Math. 87 (2001), no.~1, 85-102 
%
\bibitem{We}
Charles Weibel. Homotopy algebraic K-theory. In: Algebraic K-theory and Algebraic Number Theory. Contemp. Math., vol.~83, Amer. Math. Soc., Providence, RI, 1989, pp. 461-488
%
\bibitem{Wl}
Jaroslaw W\l odarczyk. Embeddings in toric varieties and prevarieties. J. Algebr. Geom. 2 (1993), 705-726
%
\end{thebibliography}
\end{document}